\documentclass{article}
 \usepackage{amsfonts, amsmath, amsthm, bbm, amssymb, color, indentfirst}

\newtheorem{coro}{Corollary}[section]
\newtheorem{defi}{Definition}[section]
\newtheorem{lem}{Lemma}[section]
\newtheorem{prob}{Problem}[section]
\newtheorem{prop}{Proposition}[section]
\newtheorem{thm}{Theorem}[section]
\newtheorem{rem}{Remark}[section]

\renewenvironment{abstract}{%
        \small
        \quotation
         \noindent {\bfseries \abstractname } }%
      {\if@twocolumn\else\endquotation\fi}

\def\f{\frac}
\def\d{\mathrm d}
\def\e{\mathrm e}
\def\i{\mathrm i}

\def\R{\mathbb{R}}

\def\D{\mathbb{D}}

\def\cU{\mathcal{U}}

\def\de{\delta}
\def\De{\Delta}
\def\al{\alpha}
\def\be{\beta}

\def\Ga{\Gamma}

\def\ve{\varepsilon}

\def\la{\lambda}
\def\La{\Lambda}
\def\vp{\varphi}
\def\na{\nabla}

\def\Om{\Omega}

\def\ov{\overline}
\def\pa{\partial}

\def\wh{\widehat}
\def\wt{\widetilde}

\def\LL{{L^2(\Om)}}
\def\HH{{H^2(\Om)}}

\def\dt{\d t}

\def\dr{\d r}
\def\da{\d \al}

\title{\bf
An inverse problem for distributed order time-fractional diffusion equations
}
\author{
Zhiyuan LI$^\dag$,\qquad Kenichi FUJISHIRO$^\ddag$, \qquad Gongsheng LI$^\dag$
}
\date{}

\begin{document}
\maketitle
\renewcommand{\thefootnote}{\fnsymbol{footnote}}
\footnotetext{\hspace*{-5mm} 
\begin{tabular}{@{}r@{}p{13cm}@{}} 
& Manuscript last updated: \today.
\\
$^\dag$& School of Mathematics and Statistics, Shandong University of Technology, Zibo, Shandong 255049, China
E-mail: zyli@sdut.edu.cn, ligs@sdut.edu.cn\\
$^\ddag$& National Institute of Technology, Ishikawa College, Japan. Email: fujishiro@ishikawa-nct.ac.jp
\end{tabular}}

\begin{abstract}
This paper deals with the distributed order time-fractional diffusion equations with non-homogeneous Dirichlet (Nuemann) boundary condition. We first prove the wellposedness of the weak solution to the initial boundary value problem for the distributed order time-fractional diffusion equation by means of eigenfunction expansion, which ensure that the weak solution has the classical derivatives. We next give a Harnack type inequality of the solution in the frequency domain under the Laplace transform, from which we further show a uniqueness result for an inverse problem in determining the weight function in the distributed order time derivative from point observation.

\vskip 4.5mm

\noindent\begin{tabular}{@{}l@{ }p{10cm}} {\bf Keywords } &
distributed order fractional diffusion,
inverse problem,
uniqueness, \quad
Harnack's inequality,
Laplace transform.

\end{tabular}

\vskip 4.5mm


\end{abstract}

\section{Introduction and main results}
\label{sec-intro-distri}

In this paper, we consider the following initial-boundary value problems 
(IBVPs, in short)
\begin{equation}
\label{equ-u-distri}
\left\{
\begin{alignedat}{2}
&\D^{(\mu)}_t u - \De u + p(x)u=0
&\quad& \mbox{in $\Om\times(0,T)$,}
\\
&u|_{t=0}=0 &\quad& \mbox{in $\Om$,}
\\
&u=g \mbox{ or } \pa_{\nu}u=g
&\quad& \mbox{on $\pa\Om\times(0,T)$.}
\end{alignedat}
\right.
\end{equation}
  Here $T>0$ is a fixed constant and $\Om$ is a bounded domain in $\R^d$, 
$d=1,2,3$, with a smooth boundary $\pa\Om$, which is defined e.g., by some $C^2$ functional relations. 
In \eqref{equ-u-distri}, $p\in L^\infty(\Omega)$ such that $p(x)\ge 0$ and 
$\pa_{\nu}u$ is defined by 
$\pa_{\nu}u:=\sum_{i=1}^{d}\nu_i\pa_{i}u$, where $\nu=(\nu_1,\cdots,\nu_d)$ 
is the unit outwards normal vector to $\pa\Om$.
$\D^{(\mu)}_t$ denotes a distributed order fractional derivative defined by 
$$
\D^{(\mu)}_t \vp(t) = \int_0^1 \pa^\al_t \vp(t) \mu(\al) \da,
$$
where $\pa_t^\al$ is the Caputo derivative of order $\al$:
$$
\pa^\al_t \vp(t) = 
\left\{
\begin{alignedat}{2}
&\vp(t), &\quad& \al = 0, 
\\
&\f{1}{\Ga(1-\al)}\int_0^t \f{\vp'(\tau)}{(t-\tau)^\al}\d\tau,
&\quad& 0<\al <1,
\\
&\vp'(t),&\quad& \al=1. 
\end{alignedat}
\right.
$$	
The conditions on $g$ and $\mu$ involved in $\D_t^{(\mu)}$ will be specified later in the statement of the main theorem.

The distributed order fractional derivative was firstly considered by Caputo \cite{C}. After that, this model has received great attention in applied disciplines due to the succeeding in modeling the ultraslow diffusion processes whose mean square displacement (MSD) admits the  logarithmic growth, say, MSD behaves like $\langle \De x^2\rangle\sim C\log t$ as $t\to\infty$ (e.g., \cite{MMPG08}), e.g., polymer physics and kinetics of particles moving in the quenched random force fields (see e.g. \cite{CRSG03}, \cite{N04}, \cite{SCK04} and the references therein).
Soon the mathematical researches on the analyzing the forward problem, such as the IBVPs for the diffusion equation with distributed order fractional derivatives, were growing rapidly, see, e.g. \cite{K08}, \cite{LLY-FCAA}, \cite{LLY-CMA}, \cite{L-FCAA} and the references therein. Namely, \cite{K08} investigated the properties of the fundamental solutions to the Cauchy problems for both the ordinary and the partial fractional differential equations with distributed order derivatives with $\mu\in C^1[0,1]$. The long- and short-time asymptotic behavior were discussed in detail in \cite{LLY-FCAA} by applying an argument similar to the derivation of the Watson lemma. \cite{L-FCAA} showed the uniqueness results for the IBVPs for the diffusion equation of distributed orders from an appropriate maximum principle. By using the Fourier method of variables separation, \cite{LLY-CMA} constructed an explicit solutions of the distributed-order time-fractional diffusion equations, with Dirichlet boundary conditions. Very recently, \cite{kubica} proved existence of a weak and regular solution for general uniformly elliptic operator under the assumption that the weight function is only integrable on the interval $[0,1]$.

Other than the above mentioned aspects for the forward problems where all the coefficients such as $\mu$ and $g$ in the mathematical model are known, in most instances the parameters which characterize the diffusion processes cannot be measured directly or easily, for example, as is known, the weight function $\mu$ in the model \eqref{equ-u-distri} should be determined by the inhomogeneity of the media, but it is not clear which physical law can relate the inhomogeneity to $\mu$, which requires one to use inverse problems to identify these physical quantities from some additional information that can be observed or measured practically. For this, we propose the following inverse problem.
\begin{prob}
\label{def-u-distri}
  Let $x_0\in\ov{\Om}$ be arbitrarily fixed. 
  We want to determine the weight function $\mu$ in $[0,1]$ by the overposed 
data $u(x_0,t)$, $t\in(0,T)$.
\end{prob}
Inverse problems in determining these unknown parameters in the model are not only important by itself, but also significant in its applications. However, the publications on the inverse problems to fractional diffusion equations are rather limited to the best of the authors' knowledge. As is known, if formally taking $\mu:=\sum_{j=1}^\ell q_j\de(\cdot-\al_j)$, where $\de$ is the Dirac-delta function, the distributed order time-fractional diffusion equation is degenerated into the so-called multi-term counterpart. For the multi-term case, there exists a large and rapidly growing number of publications related to the inverse problems in determining $\mu$.
\cite{HNWY} established a formula of reconstructing the order of fractional derivative in time in the fractional diffusion equation by time history at one fixed interior point. We refer to \cite{LY14} in which the authors pointed out that the one interior point observation is enough in reconstructing the unknown fractional orders, and \cite{LIY} where the uniqueness for reconstructing the unknown fractional order and potential was proved with the infinite measurement: Dirichlet-to-Neumann map. We also refer to \cite{CNYY} for the recovery of fractional order and diffusion coefficient simultaneously from one endpoint observation. The uniqueness result was proved based on the eigenfunction expansion of the weak solution to the IBVP and the Gel'fand-Levitan theory. It reveals that analyticity of the solution to the IBVPs was well performed in determining the fractional orders in the multi-term case. In the case of $\mu\in C[0,1]$, we refer to \cite{LLY-CMA} in which the uniqueness for the inverse problem in determining the weight function $\mu$ was proved after establishing the analyticity of the solution. Very recently, in one-dimensional case, \cite{RZ16} studied an inverse problem similar to that in \cite{LLY-CMA} by using the value of the solution $u$ in the time interval $(0,\infty)$. 

However, it turns out that the study on this kind of inverse problems of the recovery of the fractional orders or the continuous counterpart $\mu$ in the model \eqref{equ-u-distri} is far from satisfactory since all the publications either assume the homogeneous boundary condition (\cite{CNYY}, \cite{LIY}, \cite{LY14} and \cite{LLY-CMA}) or study this inverse problem by the measurement on $t\in(0,\infty)$ (\cite{HNWY} and \cite{RZ16}).

In this paper, by establishing a Harnack type inequality and using strong maximum principle of the elliptic equations, we generalize the result in \cite{RZ16}. Before giving the main result for our inverse problem, we first give a definition which gives the class of weight function under determination.
\begin{defi}
We call a function $\mu\in C[0,1]$ as a finite oscillation function if for any $c\in\R$, the following set $\{\al;\mu(\al)=c\}$ has at most finite number of isolated point.
\end{defi}
We then introduce an admissible set of the weight function $\mu$:
$$
\cU:=\{\mu\in C[0,1]; 
\mbox{ $\mu\ge, \not\equiv0$ and $\mu$ is a finite oscillatory function.}\}.
$$
  We further assume that the boundary values are positive and sufficiently 
smooth;
\begin{align*}
	&\mathcal{G}_1:=\{g\in C^{\infty}_{0}((0,T);H^{7/2}(\pa\Om)); 
		g\ge0\ \mbox{and}\ g\not\equiv0\}, \\
	&\mathcal{G}_2:=\{g\in C^{\infty}_{0}((0,T);H^{5/2}(\pa\Om)); 
		g\ge0\ \mbox{and}\ g\not\equiv0\}.
\end{align*}
 Now we are ready to state our main results:
\begin{thm}
\label{thm-IP}
$(a)$  Let $x_{0}\in\Om$ be a fixed point and let $u$ and $\wt u$ be the solutions 
to the problem \eqref{equ-u-distri} with respect to $(\mu,g)$ and 
$(\wt\mu,\wt g)$ in $\cU\times\mathcal{G}_1$ $($ Dirichlet boundary condition $)$.
  Then $\mu=\wt\mu$ in $[0,1]$ provided the overposed data 
$u(x_0,\cdot)=\wt u(x_0,\cdot)$ in $(0,T)$.

$(b)$ Assuming the potential $p(x)\ge c_0>0$ for some constant $c_0$. Let $x_{0}\in\partial\Om$ be a fixed point and let $u$ and $\wt u$ be the solutions 
to the problem \eqref{equ-u-distri} with respect to $(\mu,g)$ and 
$(\wt\mu,\wt g)$ in $\cU\times\mathcal{G}_2$ $($ Neumann boundary condition $)$.
  Then $\mu=\wt\mu$ in $[0,1]$ provided the overposed data 
$u(x_0,\cdot)=\wt u(x_0,\cdot)$ in $(0,T)$.

\end{thm}

\begin{rem}
In assertion (b), we can change the observation point $x_0$ to the interior point and get the same uniqueness result by a similar argument.
\end{rem}

\begin{rem}
The assumptions $\mathcal G_1$ and $\mathcal G_2$ for the boundary conditions seem unnatural and unreasonable at first glance. Actually we can relax the regularity assumption on the boundary condition $g$, and the compact support condition is also not essential for our problem but we do not discuss here. Most importantly, one should remember that the inverse problem is usually designed as an experiment where the boundary condition $g$ should be regarded as an input, and one can choose the input arbitrarily in a large enough function space. 
\end{rem}

\bigskip
  The rest of this paper is organized as follows.
  Section \ref{sec-fp} is devoted to the wellposedness of the IBVPs 
\eqref{equ-u-distri}.
  In Section \ref{sec-proof}, preparing all necessities about the solutions of 
\eqref{equ-u-distri}, say, the wellposedness and 
Harnack's inequality, we finish the proof of Theorem \ref{thm-IP}.
  Finally, concluding remarks are given in Section \ref{sec-rem}.

\section{Forward problems}
\label{sec-fp}

As is known, most of the solvability of inverse problems is very dependent 
on forward problems no matter whether it is the pure theory or numerical 
theory of the inverse problems.
  In this section, we will consider the wellposedness of the IBVPs 
\eqref{equ-u-distri}, which mainly assert the 
continuity of the solutions so that enables the measurement of solutions make 
sense at one point $x_0$. 

\subsection{Dirichlet boundary condition}
\begin{prop}
\label{thm-fp}
Let $g\in C_0^\infty((0,T);H^{\f72}(\pa\Om))$.
We assume the weight function $\mu\in C[0,1]$ is nonnegative, and not vanish 
in $[0,1]$.
  Then the problem \eqref{equ-u-distri} with Dirichlet condition $u=g$ on the boundary $\partial\Omega\times(0,T)$ admits a unique solution 
$u\in C^\infty((0,T);H^4(\Om))$, satisfying
$$
	\|u\|_{C^m([0,T];H^4(\Om))} 
\le CT\max\{1,T\}\|g\|_{C^{m+2}([0,T];H^{\f72}(\pa\Om))},\quad m=0,1,\cdots.
$$
Here the constant $C>0$ only depends on $m,\mu,d,\Om$.
\end{prop}

Before giving the proof of the above lemma, we first  introduce the Dirichlet eigensystem $\{\la_n,\vp_n\}$ of the operator $-\De + p(x)$, that is,
$$
(-\De+p(x))\vp_n=\la_n\vp_n,\quad \vp_n\in H_0^1(\Om)\cap \HH.
$$ 
For short we denote $w(s):=\int_0^1 \mu(\alpha) s^{\alpha-1} \d\alpha$ and we define an operator $I^{(\mu)}$ as
$$
I^{(\mu)} \phi(t):=\int_0^t \kappa(t-\tau) \phi(\tau)\d \tau,\quad \mbox{where } \kappa(t):=\f1{2\pi\i} \int_{\gamma-\i\infty}^{\gamma+\i\infty} \frac{1}{sw(s)}\e^{st}\d s.
$$
We now turn to considering the following IBVP
\begin{equation}
\label{eq-v}
\left\{
\begin{alignedat}{2}
&\D_t^{(\mu)} v - \De v + p(x)v = F &\quad& \mbox{in $\Om\times(0,T)$,}\\
& v|_{t=0}=0 &\quad& \mbox{in $\Om$,}\\
& v=0 &\quad& \mbox{on $\pa\Om\times(0,T)$,}
\end{alignedat}
\right.
\end{equation}
  where $F\in C^1_0((0,T);\HH)$. With reference to Corollary 3.1 in 
\cite{RZ16} and the above notations, the solution $v$ to \eqref{eq-v} can be 
represented in the form
\begin{equation}
\label{sol-v1}
v(x,t)=\sum_{n=1}^\infty \int_0^t (\pa_t  F(\cdot,\tau),\vp_n) I^{(\mu)}v_n(t-\tau)\d\tau \vp_n(x),\quad (x,t)\in \Om\times(0,T),
\end{equation}
where $(\cdot,\cdot)$ denotes the inner product in $\LL$, and $v_n(t)$ is the unique solution of the distributed ordinary differential equation
$$
\left\{
\begin{alignedat}{2}
& \D_t^{(\mu)} v_n(t) = -\la_n v_n(t), &\quad& t>0,\\
& v_n(0)=1.
\end{alignedat}
\right.
$$

  Armed with the above argument, in what following, we will give the formal 
representation of the solution to the IBVP \eqref{equ-u-distri} and give the 
proof of the conclusions stated in Proposition \ref{thm-fp}.
  For this, we introduce the operator $\La: L^2(\pa\Om)\to H^{\f12}(\Om)$ by
$$
	\La \phi(x):= 
	-\sum_{n=1}^\infty \f1{\la_n}(\phi,\pa_\nu\vp_n)_{L^2(\pa\Om)}\vp_n(x),
	\quad \phi\in L^2(\pa\Om).
$$
  It is not very difficult to see that $\La g$ solves the following 
boundary value problem for the elliptic equation
\begin{equation}
\label{equ-g}
\left\{
\begin{alignedat}{2}
& (-\De+p(x))(\La g) =0 &\quad& \mbox{in $\Om$,} \\
& \La g = g &\quad& \mbox{on $\pa\Om$.}
\end{alignedat}
\right.
\end{equation}
in view of Lemma 2.1 in \cite{F14}, or one can prove directly by integration by parts. 

As a byproduct of the regularity estimate for above boundary value problem \eqref{equ-g} (see, e.g., \cite{LM12}) and the assumption $g\in C^\infty_0((0,T);H^{\f72}(\pa\Om))$, we find that $\La g\in C_0^\infty((0,T);H^4(\Om))$ and there exists a positive constant $C$ which is independent of $t$, $T$ and $g$ such that the following estimate
\begin{equation}
\label{esti-g}
\|\pa_t^i \La g(t)\|_{H^4(\Om)} \le C\|\pa_t^i g(t)\|_{H^{\f72}(\pa\Om)},\quad t\in(0,T),\ i=0,1,\cdots
\end{equation}
holds true.

Now letting $w(x,t):=u(x,t) - (\La g)(x,t)$, we see that $w$ reads
$$
\left\{
\begin{alignedat}{2}
& \D_t^{(\mu)} w - \De w + p(x) w = -\D_t^{(\mu)} (\La g) &\quad& \mbox{in $\Om\times(0,T)$,} \\
& w|_{t=0} = u|_{t=0} - \La g|_{t=0} = 0 &\quad& \mbox{on $\Om$,}\\
& w = u - \La g = 0 &\quad& \mbox{on $\pa\Om\times(0,T)$.}
\end{alignedat}
\right.
$$
Consequently, with reference to \eqref{sol-v1}, we obtain
$$
w(x,t)=-\sum_{n=1}^\infty \int_0^t ( \pa_t \D_t^{(\mu)}\La g(\cdot,\tau),\vp_n) I^{(\mu)}v_n(t-\tau)\d\tau \vp_n,\quad t\in(0,T),
$$
and hence
\begin{equation}
\label{sol-u}
u(x,t)=(\La g)(x,t)-\sum_{n=1}^\infty \int_0^t ( \pa_t \D_t^{(\mu)}\La g(\cdot,\tau),\vp_n) I^{(\mu)}v_n(t-\tau)\d\tau \vp_n,\quad t\in(0,T).
\end{equation}
We are now ready to give the proof of Proposition \ref{thm-fp}.
\begin{proof}[\bf Proof of Proposition \ref{thm-fp}]
We shall treat each of each term on the right-hand side of \eqref{sol-u} separately. Firstly, by arguing as in the derivation of Corollary 3.1 in \cite{RZ16}, we find that $w\in C^\infty((0,\infty);\HH)$ and the following inequality
\begin{align*}
\|w\|_{C^m([0,T];H^4(\Om))} 
\le C\sum_{i=1}^{m+1}\|\pa_t^i \D_t^{(\mu)}\La g\|_{L^2(0,T;H^4(\Om))},
\quad m=0,1,\cdots
\end{align*}
is valid, which gives an evaluation for $u$: 
\begin{align*}
\|u\|_{C^m([0,T];H^4(\Om))} 
\le& C\sum_{i=0}^{m+1} \|\pa_t^i \La g\|_{L^\infty(0,T;H^4(\Om))}
+C\sum_{i=1}^{m+1}\|\pa_t^i \D_t^{(\mu)}\La g\|_{L^2(0,T;H^4(\Om))}\\
\le& C\sum_{i=0}^{m+1} \|\pa_t^i g\|_{L^\infty(0,T;H^{\f72}(\pa\Om))}
+C\sum_{i=1}^{m+1}\|\pa_t^i \D_t^{(\mu)}\La g\|_{L^2(0,T;H^4(\Om))}
\end{align*}
upon applying the above estimate \eqref{esti-g} for $\La g$.

In light of the above inequalities, it is enough to evaluate $\|\pa_t^i \D_t^{(\mu)}\La g\|_{L^2(0,T;H^4(\Om))}$ for $i=1,2,\cdots,m+1$. We start with $\pa_t \D_t^{(\mu)}\La g$. For this, as a preamble, from the definition of the distributed order time fractional derivative $\D_t^{(\mu)}$, it follows that
$$
\pa_t \D_t^{(\mu)}\La g
=\int_0^1 \f{\mu(\alpha)}{\Ga(1-\alpha)} \left[\pa_t \int_0^t \f{\pa_t\La g(\cdot,\tau)}{(t-\tau)^{\alpha} }\d\tau\right]\d\alpha
=\int_0^1 \f{\mu(\alpha)}{\Ga(1-\alpha)} \left[\int_0^t \f{\pa_t^2 \La g(\cdot,\tau)}{(t-\tau)^{\alpha} } \d\tau\right]\d\alpha,
$$
where in the last equality we used the assumption that $g\in C_0^\infty((0,T);H^{\f72}(\pa\Om))$, and hence
\begin{align*}
\|\pa_t \D_t^{(\mu)}\La g\|_{H^4(\Om)}
\le& \|\mu\|_{C[0,1]}\|\pa_t^2 \La g\|_{C([0,T];H^4(\Om))} \int_0^1 \frac1{\Ga(1-\alpha)}\left[\int_0^t \tau^{-\alpha} \d\tau\right]\d\alpha
\\
\le& \|\mu\|_{C[0,1]}\|\pa_t^2 g\|_{C([0,T];H^{\f72}(\pa\Om))} \int_0^1 \frac1{\Ga(2-\alpha)}t^{1-\alpha} \d\alpha,\quad t\in(0,T).
\end{align*}
Here in the last equality, we used the estimate \eqref{esti-g} and the property of the Gamma function $\Ga(1+\beta) = \beta\Ga(\beta)$, $\beta>0$. Now from the continuity of the Gamma function on the interval $[1,2]$, we finally derive that
\begin{align*}
\|\pa_t \D_t^{(\mu)}\La g\|_{H^4(\Om)}
\le \|1/{\Ga}\|_{L^\infty(1,2)} \|\mu\|_{C[0,1]}\|g\|_{C^2([0,T];H^{\f72}(\pa\Om))} \max\{T,1\},\quad t\in(0,T).
\end{align*}
Similarly, for $k=2,\cdots,m+1$, we have
\begin{align*}
\|\pa_t^k \D_t^{(\mu)}\La g\|_{H^4(\Om)}
\le \|1/{\Ga}\|_{L^\infty(1,2)} \|\mu\|_{C[0,1]}\|g\|_{C^{k+1}([0,T];H^{\f72}(\pa\Om))} \max\{T,1\},\quad t\in(0,T).
\end{align*}
Collecting all the above estimates leads to
\begin{align*}
\|u\|_{C^m([0,T];H^4(\Om))} 
\le CT\max\{1,T\}\|g\|_{C^{m+2}([0,T];H^{\f72}(\pa\Om))},
\end{align*}
where the constant $C$ is independent of $g$ and $T$ but may depend on $m,\mu,d,\Om$.
\end{proof}

\subsection{Neumann boundary condition}
  Next we consider the forward problem for IBVP \eqref{equ-u-distri} with Neumann boundary condition $\partial_\nu u=g$ on $\Omega$.
  We assert the following result of regularity and estimate.

\bigskip
\begin{prop}\label{thm-fp2}
  Let $g\in C_0^\infty((0,T);H^{\f52}(\pa\Om))$.
  We assume the weight function $\mu\in C[0,1]$ is nonnegative, and not vanish 
in $[0,1]$.
  Then the problem \eqref{equ-u-distri} with Neumann boundary condition admits a unique solution 
$u\in C^\infty((0,T);H^4(\Om))$, satisfying
$$
	\|u\|_{C^m([0,T];H^4(\Om))} 
\le CT\max\{1,T\}\|g\|_{C^{m+2}([0,T];H^{\f52}(\pa\Om))},\quad m=0,1,\cdots.
$$
  Here the constant $C>0$ only depends on $m,\mu,d,\Om$.
\end{prop}

\bigskip
  In order to deal with the problem with Neumann boundary condition, we 
prepare the operator $A$ on $L^2(\Om)$ defined as
\[
\begin{cases}
	D(A)=\{u\in H^2(\Om);\ \pa_{\nu}u=0\}, \\
	Au=-\De u+p(x)u,\quad u\in D(A).
\end{cases}
\]
  Then $A$ has the eigensystem $\{\mu_{n}, \psi_{n}\}$ and we can define the 
operator valued function $\{S_{A}(t)\}_{t\ge0}$ by
\[
	S_{A}(t)u=\sum_{n=1}^{\infty}(u,\psi_n)I^{(\mu)}v_n(t)\psi_n, \quad
	u\in L^2(\Om),
\]
  where $v_n(t)$ is the unique solution of the distributed ordinary 
differential equation
\[
\left\{\begin{alignedat}{2}
	& \D_t^{(\mu)} v_n(t) + \mu_n v_n(t)=0, \quad t>0,\\
	& v_n(0)=1.
\end{alignedat}\right.
\]
  We consider the following abstract problem in $L^2(\Om)$;
\begin{equation}\label{eq-Op}
\left\{\begin{alignedat}{2}
	&\D_t^{(\mu)} u(t) + Au(t) = F(t), \quad t\in(0,T), \\
	&u(0)=0, 
\end{alignedat}\right.
\end{equation}
  where $F$ is given and such that $F(0)=0$.
  We note that \eqref{eq-Op} admits a unique solution represented by
\begin{equation}\label{sol-Op}
	u(t)=\int_{0}^{t}S_{A}(t-\tau)F'(\tau)\d\tau, \quad t\in(0,T).
\end{equation}
  By repeating the argument for \eqref{equ-u-distri}, we obtain the 
representation formula of the solution to \eqref{equ-u-distri} as follows;
\begin{align}\label{sol-u2}
	u(x,t)&=(\ov{\La}g)(x,t)
	-\int_{0}^{t}S_{A}(t-\tau)(\D_t^{(\mu)}\ov{\La}g)'(\tau)\d\tau 
		\nonumber\\
	&=(\ov{\La}g)(x,t)-\sum_{n=1}^\infty \int_0^t (\pa_t \D_t^{(\mu)}\ov{\La}g(\cdot,\tau),\psi_n) I^{(\mu)}v_n(t-\tau)\d\tau \psi_n,\quad t\in(0,T),
\end{align}
  where the operator $\ov{\La}:L^2(\pa\Om)\to H^{3/2}(\Om)$ maps $g$ to the 
solution of the following boundary value problem of elliptic equation;
\begin{equation}\label{equ-g2}
\left\{\begin{alignedat}{2}
	& (-\De+p(x))(\ov{\La} g) =0 &\quad& \mbox{in $\Om$,} \\
	& \pa_{\nu}(\ov{\La} g) = g &\quad& \mbox{on $\pa\Om$.}
\end{alignedat}\right.
\end{equation}
  Moreover we can also have estimate stated in Proposition \ref{thm-fp2}.

\bigskip
  We finish this section by showing the following maximum principle for 
Neumann boundary value problem;

\bigskip
\begin{prop}\label{prop:Neumann}
  Let $u\in C^{2,1}(\overline{\Om}\times[0,T])$ satisfy the differential 
inequality below:
\begin{eqnarray}
	\D_{t}^{(\mu)}u-\De u+p(x)u\ge0
		&\mbox{in}&\ \Om\times(0,T), \label{eq:Inq} \\
	u|_{t=0}\ge 0
		&\mbox{in}&\ \Om, \label{eq:IC} \\
	\pa_{\nu}u\ge 0
		&\mbox{on}&\ \pa\Om\times(0,T), \label{eq:BC}
\end{eqnarray}
  Then we see that $u\ge 0$ in $\overline{\Om}\times[0,T]$.
\end{prop}

\bigskip
  This can be shown similarly to Lemma 2.1 in \cite{LLYZ}.
  For the proof, we prepare the following important lemma.

\bigskip
\begin{lem}[\cite{L-FCAA}]
  Let $f:[0,T]\to\mathbb{R}$ be a smooth function which attains its minimum at 
$t_{0}\in(0,T]$.
  Then
\[
	\pa_{t}^{\alpha}f(t_{0})\le 0, \quad 0<\alpha<1.
\]
\end{lem}

\bigskip
  From this lemma, we easily deduce
\begin{equation}\label{eq:min}
	\D_{t}^{(\mu)}f(t_{0})
	=\int_{0}^{1}\pa_{t}^{\alpha}f(t_{0})\mu(\alpha)d\alpha
	\le 0
\end{equation}
  by the positivity of $\mu:[0,1]\to\mathbb{R}$.

\bigskip
\begin{proof}[\bf Proof of Proposition \ref{prop:Neumann}]
  We first note that by taking $u+\ve t$ for $\ve>0$ and letting $\ve\to0$, we 
may consider the following differential inequality instead of \eqref{eq:Inq};
\begin{equation}\label{eq:Inq'}
	\D_{t}^{(\mu)}u-\De u+p(x)u>0 \quad\mbox{in}\ \Om\times(0,T).
\end{equation}
  We will prove by contradiction argument that the solution $u$ of IBVP 
\eqref{eq:Inq'}, \eqref{eq:IC} and \eqref{eq:BC} satisfy the same conclusion 
$u\ge 0$.
  To this end, we assume that $u(x,t)<0$ for some 
$(x,t)\in\overline{\Om}\times[0,T]$.
  Then, by its continuity, $u$ attains its minimum less than $0$ at some 
$(x_{0},t_{0})\in\overline{\Om}\times[0,T]$.
  That is,
\[
	\min_{(x,t)\in\overline{\Om}\times[0,T]}u(x,t)
	=u(x_{0},t_{0})<0.
\]
  Since $u(x,0)\ge 0$, we may further assume that
\[
	t_{0}>0.
\]
  Then we can show the existence of $x_{*}\in\Om$ satisfying
\begin{equation}\label{eq:Du}
	\De u(x_{*},t_{0})-p(x_{*})u(x_{*},t_{0})\ge 0
\end{equation}
  for the following two cases;
\[
	x_{0}\in\Om	\quad\mbox{and}\quad x_{0}\in\pa\Om
\]
  If $x_{0}\in\Om$, then \eqref{eq:Du} is obvious by 
taking $x_{*}=x_{0}$.
  If $x_{0}\in\pa\Om$, then we can prove 
\eqref{eq:Du} by contradiction argument as follows.
  Suppose $\De u(x_{0},t_{0})-p(x_{0})u(x_{0},t_{0})<0$.
  Then by the continuity of $u$ and $\De u$, there exists $\de>0$ such that
\[
	|x-x_{0}|<\de \implies \De u(x,t_{0})-p(x)u(x,t_{0})<0.
\]

  Now we set $\Om_{\de}:=\{x\in\Om;\ |x-x_{0}|<\de\}$.
  Then there exists $x_{1}\in\Om_{\de}$ such that
\begin{equation}\label{eq:x1}
	u(x_{1},t_{0})=u(x_{0},t_{0})
\end{equation}
  Indeed, suppose by contradiction that $u(x,t_{0})>u(x_{0},t_{0})$ holds 
for any $x\in\Om_{\de}$.
  Then by Lemma 3.4 in \cite{G-T}, we have $\pa_{\nu}u(x_{0},t_{0})<0$, 
which is impossible by boudary condition \eqref{eq:BC}.
  Thus we see that \eqref{eq:x1} is true.

  Since $(x_{1},t_{0})$ is an interior point of $\Om\times(0,T)$ and 
$u(x_{1},t_{0})$ is minimum and negative, we have
\[
	\De u(x_{1},t_{0})-p(x_{1})u(x_{1},t_{0})\ge 0
\]
  Therefore, by taking $x_{*}=x_{1}$, inequality \eqref{eq:Du} holds.

  Thus we have proved \eqref{eq:Du}.
  Combining this with \eqref{eq:min}, we deduce the contradiction;
\[
	0\le\De u(x_{*},t_{0})-p(x_{*})u(x_{*},t_{0})
	<\D_{t}^{(\mu)}u(x_{*},t_{0})\le 0.
\]
  Hence we must have $u\ge 0$.
\end{proof}

\section{Proof of Main Result}
\label{sec-proof}
On the basis of the above section for the forward problem, it reveals that the pointwise observation is possible. In this section, we turn to finishing the proof for our inverse problem stated in Theorem \ref{thm-IP} in the determination of the weight function $\mu$ from the point measurement. To this end, we first show a very useful Harnack type estimate for the elliptic equations with parameters.

\subsection{Harnack's inequality}

The following useful lemma holds:
\begin{lem}[Harnack's inequality]
\label{lem-har}
Assume $u\ge0$ is a $C^2$ solution of 
$$
-\De u + (\la+p(x)) u = 0 \mbox{ in $\Om$,}
$$
  where $\la>0$ is a constant.
  Consider a subdomain $U\subset\subset\Om$. Then 
\begin{equation}
\label{esti-har}
\sup_{U} u \le \e^{C(1+\la)}\inf_U u
\end{equation}
for some positive constant $C$ that depends on $d,U,\Om$.
\end{lem}
The proof is followed from the classical idea in deriving the Harnack inequality for parabolic equation from \cite{E10}.
\begin{proof}
Without loss of generality, we assume $u>0$ (else consider $u+\ve$, and send $\ve\to0$). Moreover, we should point out here that if we let
$$
v=\log u,
$$
then the Harnack inequality follows once we show that $\|\na v\|_{L^\infty(U)} \le C\la$. Indeed, in this case for $x_1,x_2\in U$ we have
$$
\f{u(x_2)}{u(x_1)} = \e^{v(x_2) - v(x_1)}
\le \e^{|x_1 - x_2| \|\na v\|_{L^\infty(U)}}
$$
and so \eqref{esti-har} holds true from boundedness of the domain $U$.

We conclude from the equation
$$
-\De u+ (\la+p(x)) u = 0 
$$
that
\begin{equation}
\label{equ-w}
|\na v|^2=-\De v + \la + p(x) = :\zeta.
\end{equation}
Now we point out that a sufficient condition for \eqref{esti-har} to hold is that $\|\zeta\|_{L^\infty(U)}\le C(\la+1)$. We shall prove the latter by looking at the elliptic equation obeyed by $\zeta$. To get an elliptic equation for $\zeta$, by direct calculation, we find
\begin{align*}
\pa_k\zeta &= 2\sum_{i=1}^d (\pa_iv)\pa_{ki}v,
\\
\pa_{kk}\zeta &= 2\sum_{i=1}^d (\pa_iv)\pa_{kki}v + 2\sum_{i=1}^d (\pa_{ki}v)^2.
\end{align*}
Therefore,
\begin{equation}
\label{equ-Aw}
	\De\zeta
 = 	2\sum_{i=1}^d \pa_iv\left(\sum_{k=1}^d \pa_{kki}v\right)
		 + 2\sum_{i,k=1}^d (\pa_{ki}v)^2
 = 	2\na v\cdot\na(\De v) + 2|\na^2 v|^2,
\end{equation}
  where $\na^2 v$ denotes the Hessian matrix;
\[
	\na^2 v
=\begin{pmatrix}
	\pa_{11}v & \pa_{12}v & \cdots & \pa_{1n}v \\
	\pa_{21}v & \pa_{22}v & \cdots & \pa_{2n}v \\
	\vdots	  & \vdots    & \ddots & \vdots \\
	\pa_{n1}v & \pa_{n2}v & \cdots & \pa_{nn}v 
\end{pmatrix}.
\]
  On the other hand, from \eqref{equ-w}, we have
\begin{equation}
\label{equ-w'}
\pa_k\zeta =-\pa_k(\De v) + \partial_kp(x).
\end{equation}
  Combining \eqref{equ-Aw} and \eqref{equ-w'} yields
\begin{equation}
\label{esti-ellip}
	\De \zeta+2\na v\cdot\na\zeta = 2|\na^2v|^2+2\nabla v\cdot\nabla p(x).
\end{equation}
  Let $\chi$ be a smooth cutoff function adapted to $(U,\Om)$, say, 
$\chi\in C_0^\infty(\Om)$ such that $0\le\chi\le1$ and $\chi(x)=1$ if 
$x\in U$, and define 
$$
z = \chi^4 \zeta.
$$

  The function $z$ is continuous (recall $u$ is $C^2$) and has compact support 
in $U$, so it attains its maximum at some point $x_0\in\Om$.
  At this point we have $\na z(x_0)=0$, and therefore
\begin{equation}
\label{equ-x0}
	\chi(x_0)\pa_k\zeta(x_0) = -4\pa_k\chi(x_0)\zeta(x_0).
\end{equation}
  Moreover, since $x_0$ is the maximum point $x_0\in\Om$ of $z$, we also have 
at $x_0$ that
\begin{equation}
\label{esti-z}
0\ge 	\De z + 2\na v\cdot\na z 
=	\chi^4 \left( \De\zeta + 2\na v\cdot\na\zeta\right)  + R
=	2\chi^4(|\na^2 v|^2+\nabla v\cdot\nabla p)+R,
\end{equation}
  where we have used \eqref{esti-ellip} and set
\[
	R:=12\chi^{2}|\na\chi|^{2}\zeta+4\chi^{3}(\De\chi)\zeta
	+8\chi^{3}\na\chi\cdot\na \zeta+8\chi^{3}(\na v\cdot\na\chi)\zeta.
\]
  By using \eqref{equ-x0}, we obtain the bound for $R$ as follows;
$$
	|R|\le C(\chi^2|\zeta| + \chi^3|\na v||\zeta|).
$$
  From this estimate and \eqref{esti-z}, we have
\begin{equation}
\label{eq:est-Hessian}
2\chi^4|\na^2 v|^2\le C(\chi^2|\zeta| + \chi^3|\na v||\zeta|)+2\chi^4 |\nabla v\cdot\nabla p|.
\end{equation}
Now combining \eqref{equ-w}, \eqref{esti-ellip} and \eqref{eq:est-Hessian} 
with the facts $|\De v|\le |\na^2v|$ and 
$$ 
2\chi^4|\nabla v\cdot\nabla a|\le  \chi^4|\nabla v|^2 + \chi^4 |\nabla p|^2
\le \chi^2|\nabla v|^2 + |\nabla p|^2,
$$
 we obtain
\begin{align*}
	\chi^4\zeta^2
&=	\chi^4|-\De v + \la|^2
\le 	2\chi^4|\De v|^2 + 2\chi^4\la^2 
\le 	2\chi^4|\na^2 v|^2 + 2\la^2 \nonumber \\
&\le 	C(\chi^2|\zeta| + \chi^3|\na v||\zeta|)+2\la^2 +  \chi^2 |\nabla v|^2 + |\nabla p|^2,
\end{align*}
  and noting that $\zeta:=|\na v|^2$, we further see that
\begin{equation}
\label{esti-x4w2}
\chi^4\zeta^2
\le C(\chi^2|\zeta| + \chi^3|\zeta|^{\f32}+\la^2 + |\nabla p|^2)
\end{equation}
for some constant $C$ that depends on $U,\Om$. But \eqref{esti-x4w2} shows that $\chi^2\zeta$ is bounded by $C(\la+1)$ on $\Om$, and since $\chi\equiv1$ on $U$, it also gives a bound on $\|\zeta\|_{L^\infty(U)}\le C(\la+1)$, where $C$ only depends on $U,\Om$, thereby concluding the proof.
\end{proof}

On the basis of the Harnack inequality, we can get the following corollary.
\begin{coro}
\label{coro-geq0}
Let the weight function $\mu\in C[0,1]$ be nonnegative, and not vanish in $[0,1]$. Suppose the non-negative function $u\in C_0^\infty((0,T);H^4(\Om))$ solves the following problem 
$$
\D^{(\mu)}_t u -\De u + p(x)u\not\equiv 0, \le0 \mbox{ in $\Om\times(0,\de)$,}
$$
where $0<\de<T$.
 Then for any $x\in\Om$, there exists $t_x\in(0,\de)$ such that 
$u(x,t_x)>0$.
\end{coro}
\begin{proof}
  We prove this corollary by contradiction argument.
  For this, we assume there exists $x_1\in\Om$ such that $u(x_1,\cdot)\equiv0$ 
in $(0,\de)$.
  Since $u\not\equiv0$ in $\Om\times(0,\de)$, we can choose a subdomain 
$U\subset\subset\Om$ containing $x_1$ and such that $u\not\equiv0$ on 
$\pa U\times(0,T)$.
  Moreover, from our assumption, it follows that 
$u|_{\pa U}\in C_0^\infty((0,\infty); H^{\f72}(\pa U))$ 
by letting $u|_{\pa U}=0$ outside of $(0,T)$. 
  Keeping this in mind, we introduce an auxiliary function $v$ satisfies the 
following equation
\begin{equation}
\label{equ-v'}
\left\{
\begin{alignedat}{2}
&\D^{(\mu)}_t v -\De v + p(x)v= 0  &\quad& \mbox{in $U\times(0,\infty)$,}
\\
&v|_{t=0}=u|_{t=0}=0 &\quad& \mbox{in $U$,}
\\
&v = u
&\quad& \mbox{on $\pa U\times(0,\infty)$.}
\end{alignedat}
\right.
\end{equation}
  By Proposition \ref{thm-fp}, we see that \eqref{equ-v'} admits a unique solution 
$v\in C^\infty((0,\infty);H^4(U))$ which does not vanish in $U\times(0,T)$.
  Moreover, we have
$$
\|v\|_{C([0,\infty); H^4(U))} \le C\|u\|_{C^2([0,T];H^{\f72}(\pa U))} < \infty,
$$
  and hence the Sobolev embedding theorem implies that 
$v\in C^\infty((0,\infty);C^2(\overline U))$, for $d\le3$.
  Moreover, from the Maximum principle (see, e.g., \cite{L-FCAA}), it follows 
that
$$
0\le v\le u \mbox{ in $U\times(0,\infty)$.}
$$
  Again Proposition \ref{thm-fp} entails the existence of the Laplace transform of 
the solution $v$;
$$
\wh v(x;s)=\int_{0}^{\infty}v(x,t)\e^{-st}\d t.
$$
  Then taking the Laplace transforms on both sides of \eqref{equ-v'} derives
$$
\left\{
\begin{alignedat}{2}
&-\De \wh v(x;s) + sw(s)\wh v(x;s) + p(x) \wh v(x;s) = 0
&\quad& \mbox{in $U$,}
\\
&\wh v(x;s)=\wh u(x;s) &\quad& \mbox{in $\pa U$.}
\end{alignedat}
\right.
$$
  Since the function $v$ does not vanish in $U\times(0,T)$, we choose 
$(x_2,t_2)\in U\times(0,T)$ such that $v(x_2,t_2)>0$.
  From the Harnack inequality proved in Lemma \ref{lem-har}, for the subdomain 
$V$ such that $V\subset\subset U$ and $x_1,x_2\in V$, we have
$$
\sup_{V}\wh v(x;s) \le \e^{C(1+|sw(s)|)} \inf_{V} \wh v(x;s),\quad s>0.
$$
From the choice of $x_1,x_2$ and $V$, we see that
$$
\wh v(x_2;s) 
\le \sup_{V}\wh v(x;s) \le \e^{C(1+|sw(s)|)} \inf_{V} \wh v(x;s)
\le \e^{C(1+|sw(s)|)} \wh v(x_1;s),\quad s>0.
$$
  Now since $u(x_1,t)=0$ for $t\in(0,\de)$, and $v\le u$, we note that 
$v(x_1,t)=0$ if $t\in(0,\de)$, which implies that
$$
	\wh v(x_1;s)
 = 	\int_{\de}^\infty v(x_1,t)\e^{-st}\dt
\le 	Cs^{-1}\e^{-\de s},\quad s>0.
$$
On the other hand, we have 
$$
	\wh v(x_2;s) \ge \int_{t_2-\eta}^{t_2+\eta} v(x_2,t) \e^{-st}\dt
\ge 	c_1s^{-1}\e^{-(t_2-\eta)s},\quad s>0.
$$
  Here $c_1:=\inf_{t\in(t_2-\eta,t_2+\eta)} v(x_2,t) >0$.
  Combining all the estimates, we find
$$
	\e^{-(t_2-\eta)s} \le \e^{C(1+|sw(s)|)} C\e^{-\de s},\quad s>0.
$$
  Here $t_2\in(0,\de)$.
  Moreover, from the notation of $sw(s):=\int_0^1 s^\al\mu(\al)\da$, we have
$$
	|sw(s)| 
\le 	\|\mu\|_{C[0,1]}|\int_0^1 \e^{\al\log s}\da
\le 	\|\mu\|_{C[0,1]}\f{s-1}{\log s},\quad s>0,
$$
  which implies
$$
	\e^{-(t_2-\eta)s} \le \e^{C(1+\f{s-1}{\log s})} \e^{-\de s},\quad s>1.
$$
  Letting $s\to\infty$, we get a contradiction in view of $t_2-\eta<\de$ and 
$1/\log s<<\de$ if $s$ is sufficiently large.
  We must have that for any $x\in\Om$ there exists $t_x\in(0,\de)$ such that 
$u(x,t_x)>0$.
  This completes the proof of the corollary.
\end{proof}

\bigskip
\begin{coro}
\label{coro-geq0'}
  Let the weight function $\mu\in C[0,1]$ be nonnegative, and not vanish in 
$[0,1]$.
  Suppose the non-negative function $u\in C_0^\infty((0,T);H^4(\Om))$ solves 
the following problem 
\begin{align*}
	\D^{(\mu)}_t u -\De u + p(x)u \le,\not\equiv 0 \mbox{ in $\Om\times(0,\de)$,}\\
	\pa_{\nu}u=0\mbox{ on $\pa\Om\times(0,\de)$}
\end{align*}
where $0<\de<T$.
 Then for any $x\in\pa\Om$, there exists $t_x\in(0,\de)$ such that 
$u(x,t_x)>0$.
\end{coro}

\bigskip
  By the boundary condition $\pa_{\nu}u=0$, we can extend the domain $\Om$ to 
$\Om'$ which contains $x_{0}\in\pa\Om$ as its inner point.
  Then we can take a subdomain $U$ such that $x_{0}\in U\subset\subset\Om'$ 
and repeating the same argument as in the proof of Corollary \ref{coro-geq0}, 
we will see that the above assertion holds true.

\subsection{Proof of Theorem \ref{thm-IP}}
In this part, we further fix some general settings and notations. We introduce the Riemann-Liouville fractional integral operator $J^\al$:
$$
J^\al \vp :=\f1{\Ga(\al)} \int_0^t \f{\vp(r)}{(t-r)^{1-\al}} \dr,\quad t>0,
$$
where $\al>0$. We see that $J^\al$ admits the semigroup property
$$
J^\al J^\be = J^{\al+\be},\quad \al>0,\ \be>0,
$$
(see, e.g., \cite{P98}). Now we are ready to give the proof.
\begin{proof}[\bf Proof of Theorem \ref{thm-IP}, Dirichlet boundary condition]

  First, in view of the fact that $d\le3$, we conclude from Proposition \ref{thm-fp} 
that $u\in C_0^\infty((0,T);C^2(\overline\Om))$.
  Moreover, the non-negativity of the function $g$ combined with the maximum 
principle (see, e.g., \cite{L-FCAA}) yields that $u(t)\ge0$ in $\Om$ for any 
$t>0$.
  The same property holds true for the solution $\wt u$ to the IBVP 
\eqref{equ-u-distri} with weight function $\wt\mu$.

  Taking the operator $J^2$ on both sides of the equation 
\eqref{equ-u-distri}, and noting that $u(0)=0$ implies 
$J^2\pa_t^\al u=\pa_t^\al J^2u=J^{2-\alpha}u$, we find
$$
\left\{
\begin{alignedat}{2}
&\D^{(\mu)}_t (J^2u) - \De (J^2u) + p(x) J^2u=0
&\quad& \mbox{in $\Om\times(0,T)$,}
\\
&J^2u|_{t=0}=0 &\quad& \mbox{in $\Om$.}
\end{alignedat}
\right.
$$

  Now by taking the difference of the above systems of $J^2u$ and $J^2\wt u$, 
it turns out that the system for $v:= J^2u - J^2\wt u$ reads
$$
\left\{
\begin{alignedat}{2}
&\D^{(\mu)}_t v -\De v + p(x)v = \int_0^1 (\wt\mu(\al) - \mu(\al)) J^{2-\al}\wt u\da &\quad& \mbox{in $\Om\times(0,T)$,}
\\
&v|_{t=0}=0 &\quad& \mbox{in $\Om$.}
\end{alignedat}
\right.
$$

Next, we will use a contradiction argument to finish the proof. For this, we assume that $\mu\not\equiv\wt\mu$. More precisely, since $\mu$ is a finite oscillatory function, without loss of generality, we can assume that there exists $\al_0\in(0,1]$ such that $\wt\mu(\al)<\mu(\al)$ for $\al\in[\al_0-4\ve,\al_0)$ and $\wt\mu(\al)=\mu(\al)$ if $\al\in(\al_0,1]$. Then we can assert that the following inequality
$$
RHS:=\int_0^1 (\wt\mu(\al) - \mu(\al)) J^{2-\al}\wt u\da \le,\not\equiv0\mbox{ in $\Om$}
$$
is valid for any $0<t<\de$ with a sufficiently small constant $\de>0$. Indeed, notice that $\mu,\wt\mu$ is continuous on $[0,1]$, since $\wt\mu(\al)<\mu(\al)$ for $\al\in[\al_0-4\ve,\al_0)$, one can choose a constant $c_0$ and sufficiently small constant $\ve>0$ such that
$$
\mu(\al) - \wt\mu(\al)  > c_0>0,\quad \al\in(\al_0-3\ve,\al_0-\ve),
$$
which combined with the fact $\widetilde u\ge0$ implies that
\begin{align*}
	RHS
=&	\left(\int_0^{\al_0-3\ve} + \int_{\al_0-3\ve}^1\right)
	( \wt\mu(\al) - \mu(\al) ) J^{2-\al}\wt u\da \\
\le& 	c_1 \int_0^{\al_0-3\ve} J^{2-\al}\wt u\da
		-c_0\int_{\al_0-2\ve}^{\al_0-\ve} J^{2-\al}\wt u\da,
\end{align*}
  where $c_1:=\|\mu\|_{C[0,1]}+\|\wt\mu\|_{C[0,1]}$.
  By changing the variable as
\[
	\al\mapsto
	(\al_0-2\ve)\left(1-\f{\al}{\al_0-3\ve}\right)
	 + (\al_0-\ve) \f{\al}{\al_0-3\ve},
\]
  the above inequality can be rephrased as follows
\begin{align*}
RHS \le \int_{\al_0-2\ve}^{\al_0-\ve} \left( \f{c_1(\al_0-3\ve)}{\ve}J^{2- \beta} \wt u - c_0J^{2-\al}\wt u \right) \da,
\end{align*}
  where we have set $\beta:=\f{(\al_0-3\ve)(\alpha-\al_0+2\ve)}{\ve}$.
  Moreover, again from the semigroup property of the Riemann-Liouville fractional operator $J^\alpha$, it follows that
\begin{align*}
RHS \le \int_{\al_0-2\ve}^{\al_0-\ve} \left( \f{c_1(\al_0-3\ve)}{\ve}J^1J^{\alpha-\beta}J^{1- \alpha} \wt u - c_0J^{2-\al}\wt u \right) \da.
\end{align*}
Again by the non-negativity of $\widetilde u$ and noticing the definition of $J^\alpha$, $\alpha>0$, we have
$$
J^1J^{\alpha-\beta}J^{1-\al}\wt u(t)
=\|J^{\alpha-\beta}J^{1-\al}\wt u\|_{L^1(0,t)}
=\f{1}{\Ga(\alpha-\beta)}\left\| \int_0^t (t-\tau)^{\alpha-\beta-1} J^{1-\al}\wt u(\tau)\d\tau \right\|_{L^1(0,t)},
$$
which entails
$$
J^1J^{\alpha-\beta}J^{1-\al}\wt u(t)\le \f{t^{\alpha-\beta}}{\Ga(1+\alpha-\beta)}  \int_0^t J^{1-\al}\wt u(\tau)\d\tau
=\f{t^{\alpha-\beta}}{\Ga(1+\alpha-\beta)} J^{2-\al}\wt u(t)
$$
upon applying the Young inequality and $\Ga(1+\gamma)=\gamma\Ga(\gamma)$, $\gamma>0$. Here in the last equality we again used the definition of the Riemann-Liouville fractional integral. Moreover, since $\alpha\in[\al_0-2\ve,\al_0-\ve]$, we further see that 
$$
0<2\ve
\le \alpha - \beta 
\le \al_0-2\ve <1,
$$
which implies that $t^{\alpha-\beta} \le t^{2\ve}$ if $t\in(0,1)$, and from the continuity of the Gamma function on the interval $[2\ve,1]$, it then follows that
$$
J^1J^{\alpha-\beta}J^{1-\al}\wt u(t)
\le c_2t^{2\ve} J^{2-\al}\wt u(t),\quad 0<t<1,
$$
where $c_2:=\|1/{\Ga(\cdot)}\|_{C[2\ve,1]}$, and combining all the above estimates, we find 
$$
RHS \le \int_{\al_1-\ve_1}^{\al_1} \left( \f{c_1c_2(\al_0-3\ve)}{\ve} t^{2\ve} - c_0 \right) J^{2-\al}\wt u\da.
$$
Therefore, choosing $\de>0$ sufficiently small such that $\f{c_1c_2(\al_0-3\ve)}{\ve} \de^{2\varepsilon} = c_0$, and then for $0<t<\de$, we can assert that 
$$
RHS\le,\not\equiv0 \mbox{ in $\Om$.}
$$
We finally obtain
$$
\left\{
\begin{alignedat}{2}
&\D^{(\mu)}_t v -\De v + p(x)v \le,\not\equiv 0  &\quad& \mbox{in $\Om\times(0,\de)$,}
\\
&v|_{t=0}=0 &\quad& \mbox{in $\Om$.}
\end{alignedat}
\right.
$$
From Corollary \ref{coro-geq0}, we assert that for any $x\in\Om$, there exists $t_x\in(0,\de)$ such that $v(x,t_x)>0$, that is
$$
J^2u(x,t_x) - J^2\wt u(x,t_x) > 0,\quad x\in\Om
$$
in view of the notation of $v$. This is a contradiction since the overposed data $u(x_0,\cdot)=\wt u(x_0,\cdot)$ in $(0,T)$ implies that $J^2u(x_0,t)=J^2{\wt u}(x_0,t)$ for any $0<t<T$. By contradiction, we must have
$$
\mu(\al)=\wt\mu(\al),\quad \al\in[0,1].
$$
This completes the proof of the theorem for the Dirichlet boundary condition case.
\end{proof}
\begin{proof}[\bf Proof of Theorem \ref{thm-IP}, Neumann boundary condition]
Similarly to the previous proof, we find that $v:= J^2u - J^2\wt u$ satisfies
$$
\left\{
\begin{alignedat}{2}
&\D^{(\mu)}_t v -\De v +p(x)v= \int_0^1 (\wt\mu(\al) - \mu(\al)) J^{2-\al}\wt u\da &\quad& \mbox{in $\Om\times(0,T)$,}
\\
& \pa_{\nu}v=0 &\quad& \mbox{on $\pa\Om\times(0,T)$,}
\\
&v|_{t=0}=0 &\quad& \mbox{in $\Om$.}
\end{alignedat}
\right.
$$
  Then applying Corollary \ref{coro-geq0'} and repeating the similar argument,
we can reach a contradiction and prove Theorem \ref{thm-IP}.
\end{proof}

\bigskip
\begin{rem}
Our method cannot work in the case when $\mu$ is not a finite oscillatory function. However, if we use the measurement data similar to the one in \cite{RZ16}, say, $u(x_0,\cdot)$ in $(0,\infty)$, it is expected to obtain the uniqueness of the inverse problem in determining the weight function.
\end{rem}
Indeed, since we now consider all the problem in the infinite time interval $(0,\infty)$, we can employ the Laplace transform $\wh\cdot(s)$ on both sides of the equation \eqref{equ-u-distri} with respect to $\mu,\wt\mu\in C[0,1]$, we find
$$
\left\{\begin{alignedat}{2}
& sw(s) \wh u(s) -\De \wh u(s) + p(x)\wh u(s)=0 &\quad& \mbox{in $\Om$,}
\\
& \wh u(s)=\wh g(s), &\quad& \mbox{on $\pa\Om$,}
\end{alignedat}\right.
$$
and 
$$
\left\{\begin{alignedat}{2}
& s\wt w(s) \wh{\wt u}(s) -\De \wh{\wt u}(s)  + p(x)\wh u(s) =0 &\quad& \mbox{in $\Om$,}
\\
& \wh{\wt u}(s)=\wh g(s) &\quad& \mbox{on $\pa\Om$,}
\end{alignedat}\right.
$$
for any $s>0$, where 
$$
w(s):=\int_0^1 s^\al \mu(\al) \da,\quad 
\wt w(s):=\int_0^1 s^\al \wt\mu(\al)\da.
$$
From the strong maximum principle for the elliptic equations, we see that $\wh u(s)$ and $\wh{\wt u}(s)$ are strictly positive in the domain $\Om$ for any $s>0$.

Now by taking the difference of the above systems, it turns out that the system for $v:=\wh u-\wh{\wt u}$ reads
$$
\left\{\begin{alignedat}{2}
& sw(s) v(s) -\De v(s)  + p(x)\wh u(s)
=- s\wh{\wt u}(s)(w(s)-\wt w(s)) &\quad& \mbox{in $\Om$,} \\
& \wh v(s)=0 &\quad& \mbox{on $\pa\Om$.}
\end{alignedat}
\right.
$$
We prove by contradiction. Let us assume that $\mu\ne\wt\mu$ in $[0,1]$ and we claim that there exists $s_0>0$ such that
$$
\int_0^1 s_0^\al (\mu(\al) - \wt\mu(\al)) \da \ne 0.
$$
This can be done by an argument similar to the proof of Theorem 2.2 in \cite{LLY-CMA}. Without loss of generality, we assume that 
$$
\int_0^1 s_0^\al (\mu(\al) - \wt\mu(\al)) \da > 0.
$$
Then we conclude from the strong maximum principle for the elliptic equations that $v(x;s_0)>0$ in $\Om$ because of $\wh{\wt u}(x;s_0)\ge0$, hence that
$$
\wh u(x;s_0) > \wh{\wt u}(x;s_0),\quad x\in\Om.
$$
This is a contradiction since $u(x_0,\cdot)=\wt u(x_0,\cdot)$ in $(0,\infty)$ implies that $\wh u(x_0;s)=\wh{\wt u}(x_0;s)$ for any $s>0$. By contradiction, we must have $\mu=\wt\mu$ in $[0,1]$.

\section{Concluding remarks}
\label{sec-rem}
In this paper, the initial-boundary value problem for the diffusion equation with distributed order derivatives was investigated. On the basis of eigenfunction expansion, we first gave a representation formula of the solution via Fourier series and showed the convergence as well as several estimates for the solution. In Proposition \ref{thm-fp}, we can relax the regularity of $g$ via the argument used in \cite{F14} but we do not discuss here. 

For the inverse problem, on the basis of Laplace transform, we first transferred the time-fractional diffusion equation to the corresponding elliptic equation with the Laplacian parameter and established a Harnack type inequality for this elliptic equation, which were further used to imply the uniqueness of the inverse problem in determining the weight function $\mu$ from one point observation provided the unknown weight function $\mu$ lies in the admissible set $\cU$. 

It should be mentioned here that the proof of the above uniqueness result heavily relies on the setting of finite oscillation function. The inverse problem in determining the weight function $\mu$ in general case, say, $\mu\not\in\cU$, by the overposed data $u(x_0,\cdot)$ in $(0,T)$  remains open. On the other hand, it will be more challenged to consider the inverse problem for the problem \eqref{equ-u-distri} with general elliptic operator. 


\end{document}